\documentclass[11pt]{amsart}

\usepackage{latexsym}

\usepackage{amsmath,amssymb,amsthm,amsfonts,latexsym}
\usepackage{pstricks, pst-node, pst-text, pst-3d}
\usepackage[bookmarks]{hyperref}
\hypersetup{backref, colorlinks=true}

\let\=\partial

\newtheorem {prop}{Proposition}[subsection]
\newtheorem {thm}[prop]{Theorem}
 \newtheorem {cor}[prop]{Corollary}
\newtheorem{lem}[prop]{Lemma}

\theoremstyle{definition}
 \newtheorem {rk}[prop]{Remark}
\newtheorem {df}[prop]{Definition}
\newtheorem {dfs}[prop]{Definitions}
\newtheorem {ex}[prop]{Example}

\newcommand{\Q} {\mathbb{Q}}
\newcommand{\vb}{\overline{v}}

\newcommand{\R} {\mathbb{R}}

\newcommand{\N} {\mathbb{N}}

\newcommand{\F}{\mathcal{F}}
\newcommand{\la}{\mathcal{L}}

\title {{\bf   Vanishing homology}}



\author{Guillaume Valette}

\address
{Department of Mathematics, University of Toronto, 40 St. George
st, Toronto, ON, Canada M5S 2E4 }

\email{gvalette@math.toronto.edu}


\thanks{}


\begin{document}
\maketitle

\begin{abstract}
In this paper we introduce a new homology theory devoted to the
study of families such as semi-algebraic or subanalytic families
and in general to any family definable in an o-minimal structure
(such as Denjoy-Carleman definable  or $ln-exp$ definable sets).
The idea is to study the cycles
 which are vanishing when we approach a special fiber. This also enables us to derive local metric invariants for germs of
definable sets. We prove that the homology groups are finitely
generated.
\end{abstract}

\section{Introduction} 

The description of the topology of a set nearby a singularity is a
primary focus of   attention of algebraic geometers.  We can
regard a semi-algebraic singular subset of $\R^n$ as a metric
subspace. Then the behavior of the metric structure of a
collapsing family reflects  implicit information on the geometry
of the singularity of the underlying set which is much more
accurate than the one provided by the study of the topology.

 In \cite{v1}, the author proved a bi-Lipschitz version of
Hardt's theorem \cite{h}. This theorem pointed out  that
semi-algebraic bi-Lipschitz equivalence is a good notion of
equisingularity to classify  semi-algebraic subsets from the
metric point of view. For this purpose, it is also very helpful to
find invariants such as homological invariants.

 In this paper we introduce a homology theory for  families of subsets which provides information about the
behavior of the metric structure of the fibers when we approach a
given fiber.
 This  enables us to construct local
metric invariants for singularities. We prove that these homology
groups are finitely generated when the family is definable in an
o-minimal structure. This allows, for instance, to define an Euler
characteristic which is a metric invariant for germs of algebraic
or analytic  sets.

In   \cite{gm1}, M. Goresky and R. MacPherson introduced
intersection
 homology and  showed that their theory satisfies Poincar\'e
 duality for pseudo-manifolds which cover a quite large class of singular sets
 and turned out to be of great interest. They also managed to compute
  the intersection homology groups from a triangulation which yields that they are finitely
 generated. 
In \cite{bb1} L.
 Birbrair and J.-P. Brasselet  define  their  admissible
 chains to construct the metric homology groups. Both theories select some chains by putting conditions on the support of the chains.
Our approach is
 similar  in the sense that our homology groups will depend on
 a {\it velocity} which estimates the rate of vanishing of the
 support of the
 chains.


 Our method relies on the result of
 \cite{v1}, where the author showed existence of a triangulation
 enclosing  the metric type of a definable singular set. To  compute the vanishing homology groups  we will not   use   the
  triangulation constructed in \cite{v1} but    Proposition
  $3.2.6$
  of the latter paper (which was actually the main step of the
  construction). It makes it possible for  the results proved below to go over  non necessarily polynomially bounded o-minimal
  structures.
It seems that the method of the present paper could be generalized
to prove that the metric homology groups  introduced in \cite{bb1}
are finitely generated as well.



\medskip


It is well known that, given a definable family, we may always
study the evolution of the fibers by studying what is called by
algebraic geometers ``the generic fiber" 
(see example \ref{ex fam} for a precise definition).

Therefore  if we carry out a homology theory for definable subsets
in an o-minimal structure expanding a given    arbitrary  real
closed field, we will have a homology theory for families.  This
is the point of view   of the
  present paper. Hence,
even for families of subsets of  $\R^n$, the case of
  an arbitrary real closed field will be required.
Our approach   will be patterned on the one of the classical
homology groups as much as possible. Some statements (Theorem
\ref{thm simplicial is singular}) are close to those given by
Goresky and MacPherson for intersection homology but of course the
techniques are radically different since the setting is  not the
same.

The admissible chains  depend on a velocity which is a convex
subgroup $v$  of our real closed field $R$. For instance, if $R$
is the field of real algebraic Puiseux series endowed with the
order making the indeterminate $T$ smaller than any positive real
number, $v$ may be the subgroup
\begin{equation}\label{eq intro def groupe}\{x:\exists N \in \N,
|x| \leq N T^2\}.\end{equation} The {\it $v$-admissible chains}
are the chains having a ``$v$-thin" support. Roughly speaking, if
$v$ is as above, {\it $v$-thin} subsets of $R^n$ are the generic
fibers of families of sets whose fibers collapse onto a lower
dimensional subset with at least the velocity $N t^2$ (if $t$ is
the parameter of the family, $N \in \N$). For instance, let us
consider the cycle given by Birbrair and Goldshtein's example.
Namely, the subset of $X \subset R^{4}$ defined by:
 \begin{eqnarray}\label{eq birgol}x^2
_1 + x^2 _2 &=& T^{2p}, \nonumber \\x^2 _3 + x^2 _4 &=& T^{2q}.
\end{eqnarray}
 This set is the generic fiber of a family of tori, such that the support of the  generators   of $H_1 (X)$  collapse onto a point at rate
$T^p$ and $T^q$ respectively.  Therefore, if for instance $p=0$
and $q=2$ then the $0$-fiber is a circle and this family of torus
is $v$-thin (with $v$ like in (\ref{eq intro def groupe})).

 Taking all the $v$-admissible chains of a definable set
$X$, we get a chain complex which immediately gives rise to the
{\it $v$-vanishing homology groups} $H_j ^v(X)$.  We will show
that these groups are finitely generated (Corollary \ref{cor
groups finitely gen}).

 If $X$ is the set  defined by (\ref{eq
birgol}) with $v$ like in (\ref{eq intro def groupe}), the
$v$-vanishing homology  groups depend on  of $p$ and $q$.
  For instance, we will prove (see Example \ref{ex tore}) that if $p=0$ and $q=2$:
$$H_1 ^{v} (X)=\Q$$
(if  $\Q$ is our coefficient group), and $H_2 ^{v} (X)=\Q$.

We may summarize it by saying that we get all the $T^2$-thin
cycles of  $X$. The group $H_j ^v(X)$  is not always  a subgroup
of  $H_j (X)$. In general we may also have cycles that do not
appear in the classical homology groups, i. e. which are in the
kernel of the natural map $H_j ^{v} (X) \to H_j (X)$. The
following picture illustrates an example for which such a
situation occurs:

\begin{figure}[ht]
\begin{pspicture}(1,-4.5)(8,-1)
\psccurve[linewidth=1pt](5,-2)(4,-1.3)(2.5,-2.3)(4,-3.3)(5,-2.6)(6,-3.3)(7.5,-2)(6,-1.3)
\psellipse[linewidth=.8pt](5,-2.3)(.1,.3) \rput(5,-4){$a
$}\rput(5,-3.7){\psline[linewidth=.8pt]{->}(0,0)(0,1)}
\end{pspicture}
\caption{}
\end{figure}

The cycle $a$ is collapsing onto a point faster than the set
itself is collapsing.
 We see that we have an admissible  one dimensional chain
$a$ which bounds a two dimensional chain which may fail to be
admissible (depending on the velocity $v$). Therefore $H_1 ^v (X)
\neq 0$ (while $H_1(X) =0$).
\medskip

%

This homology theory is not a homotopy invariant. It is preserved
by Lipschitz homotopies but these are very hard to construct. For
instance, given a function $f: \R^n \to \R$ it is well known that
there exists a topological deformation retract of
$f^{-1}(0;\varepsilon)$ onto $f^{-1}(0)$. It is easy to see that
it is {\it not} possible to find such a retract which would be
Lipschitz if $f(x;y)=y^2-x^3$.  The method used in this paper
provides homotopies that are not Lipschitz but which preserve
admissible chains. It seems that one could define various homology
theories for which this method could be adapted. The theory
developed below seemed to the author the simplest one and the most
natural to start.

We compute the vanishing homology groups in terms of some basic
sets obtained by constructing some nice cells decompositions
(Theorem \ref{thm simplicial is singular}). For this we construct
a homotopy which carries a given singular chain to a chain of
these basic sets (Proposition \ref{prop homotopie}). The homotopy
has to preserve thin subsets. We are not able to construct such a
homotopy for any admissible chain. Chains for which we can
construct such a homotopy are called strongly admissible and are
chains for which the distances in the support are known in a very
explicit way. Therefore, the first step is to show that any class
in $H_j ^v (X)$ has a strongly admissible representant (Lemma
\ref{lem strong adm is onto}). This is achieved by constructing
some rectilinearizations of $v$-thin sets (Proposition \ref{prop
cell dec}). These are maps which transform our set into a union of
hyperplanes crossing normally while controlling the distances in
the transformation.

A non trivial convex subgroup $v$  may be regarded as an interval
in $R$ which has no endpoint. This fact will somewhat complicate
our task. To overcome this difficulty, we introduce an extra point
$u$ ``at the end of $v$" which will fill the gap. This point
living in an extension $k_v$ of $R$, we will carry out most of the
constructions rather in $k_v$ than in $R$. The precise definition
of $k_v$ and the basic related notions are provided in the first
section below. An advantage of using model theory is that we are
able to carry out the theory for all the possible velocities (see
example \ref{ex 1}) in the same time.

\medskip

We may use these homology groups to derive invariants for
semi-algebraic  singularities.  Given a germ $A$  of
semi-algebraic subset of $\R^n$ at the origin, the link of $A$ is
the subset
$$L_r:=A \cap B(0;r)$$
for $r$ small enough. It is known that the homology of the latter
set is a topological invariant of $A$. The cycles of $L_r$ are
collapsing to a single point with a certain ``rate". This rate is
related to
the metric type of the singularity. 

 It is proved in \cite{v2} that the metric type of the generic fiber of the family $L_r$, namely $L_{0_+}$, is a
 metric invariant of $A$. Therefore the vanishing homology groups $H_j ^v (L_{0_+})$  are
  semi-algebraic
 bi-Lipschitz invariants of $A$ (see section 4.).

%


\medskip


\noindent {\bf Content of the paper.} In section $1$, we provide
all the basic definitions about the vanishing homology.  We  prove
in the next section some cell decomposition theorems and
rectilinearization theorems necessary to compute the vanishing
homology groups. In section $3$, we compute the $v$-vanishing
homology groups in terms of this cell decomposition. The main
result is Theorem \ref{thm simplicial is singular} which yields
that the homology groups are finitely generated. In section $4$ we
give an application: we find local metric invariants for
singularities. The last section computes the vanishing homology
groups on some examples.

\medskip

 The reader is referred   to \cite{coste} or \cite{vds2} for basic
facts about o-minimal structures.

{\bf Notations and conventions.} Throughout this paper we work
with  a fixed o-minimal structure expanding a real closed field
$R$. Let $\la_R$ be the first order language of ordered fields
together with an $n$-ary function symbol for each function of the
structure. The word definable means $\la_R$-definable. The
language $\la_R(u)$ is the language $\la_R$ extended by an extra
symbol $u$.

The letter $G$ will stand for an abelian group (our coefficient
group). {\it Singular simplices} will be definable continuous maps
$c: T_j \to X$, $T_j$ being the $j$-simplex spanned by
$0,e_1,\dots,e_j$ where $e_1,\dots, e_j$ is the canonical basis of
$R^j$. Sometimes, we will work in an extension $k_v$ of $R$ and
simplices will actually be  maps $c: T_j(k_v) \to k_v^n$ where
$T_j(k_v)$ is the extension of $T_j$ to $k_v$. Given a definable
set $X \subset R^n$ we denote by $C(X)$ the chain complex of
definable chains  with coefficients in a given group $G$. We will
write $|c|$ for the support of a chain $c$.
\medskip

By {\bf Lipschitz function} we will mean a function $f$ satisfying
$$|f(x)-f(x')| \leq N|x-x'|$$
for some integer $N$. It is important to notice that we require
the constant to be an integer for $R$ is not assumed to be
archimedean. A map $h:A \to R^n$ is Lipschitz if all its
components are, and a homeomorphism $h$ is bi-Lipschitz if $h$ and
$h^{-1}$ are Lipschitz.

We denote by $\pi_n: R^n \to R^{n-1}$  the canonical projection
and by $cl(X)$ the closure of a definable set $X$.

%
%

\bigskip
\section{Definition of the vanishing homology.}
\subsection{The velocity $v$.} We shall use some very basic facts of model theory. We refer the reader
to \cite{m} for basic definitions.

 The vanishing homology  depends on a {\bf velocity}
$v$ which   estimates the rate of vanishing of the cycles. This is
a  convex subgroup $v$  of $(R;+)$ (convex in the sense that it is
a convex subset of $R$).

We then   define a  $1$-type  by saying that a sentence $\psi(u)
\in  \mathcal{L}_R(u)$ is in this type  iff the set
$$\{x \in R:\psi(x)\}$$
contains an  interval $[a;b]$ with $a \in v$ and $b \notin v$.
This type is complete due to the o-minimality of the theory.


 We will denote by $k_v$ an $\la_R$-elementary extension of $R$ realizing this type.

  Roughly speaking we can say that the velocity is
 characterized by a cut in $R$, at which  the gap is ``bigger" than
 the distance to the origin. This is to ensure that  the sum of two admissible chains will be admissible  (see section 1.3).

\noindent {\bf Notations.}  Throughout this paper, a velocity $v$
is fixed and $u$ is the point realizing the corresponding type in
$k_v$.

  We  define a convex
subgroup $w$ of $(k_v;+)$ extending the group $v$ in a natural
way: $$w:=\{x \in k_v: \exists\, y \in v,  |x|\leq y\}.$$

\begin{rk}\label{rk velocity et points}
Given $z \in R$ we may define a velocity $\N z$ by setting: $$\N z
:= \{ x \in R:\exists\, N \in \N, |x| \leq Nz\}.$$
\end{rk}

\begin{ex}\label{ex 1} Let $k(0_+)$ be the field of real algebraic Puiseux series endowed with the order
 that makes the indeterminate $T$ positive and smaller than any real number (see \cite{bcr} example 1.1.2).  Then, as in the above remark,  the element
  $T^k$ gives rise to a  subgroup $\N T^k$
 which is constituted by  all the series $z$ having a valuation greater or equal to $k$.  
  One could also consider the velocity
 $v$ defined
 by  the set
 of $x$ satisfying $|x| \leq N T^k$ for {\it any} $N$ in  $\N$. In the field of $ln-exp$ definable germs of one variable functions
 (in a right-hand side neighborhood)
one may consider  the set of all the $L^p$ integrable germs of
series.\end{ex}

%

\noindent {\bf Extension of functions.}  On the other hand, as
$k_v$ is an elementary extension of $R$, it is well known that we
may define $X_v$, the extension of $X$ to $k_v$, by regarding the
formula defining $X$ in $k_v ^n$. Every mapping $\sigma: X \to Y$
may also be extended to a mapping $\sigma_v : X_v \to Y_v$.

\subsection{$v$-thin sets.}
We give the definition of the $v$-thin sets which is required to
introduce the vanishing homology.

\begin{dfs}
Let   $j \leq n$ be  integers. A $j$-dimensional definable subset
$X$ of $R^n$ is called {\bf $ v$-thin} if there exists $z \in v$
such that, for any linear projection $\pi:R^n \to R^j$, no ball
(in $R^j$) of radius $z$ entirely lies in
$\pi(X)$. 

 For simplicity we say that $X$ is {\bf$(j;v)$-thin}  if either $X$ is  $v$-thin  or  $\dim X<j$.
 A set which is not $v$-thin will
 be called {\bf $v$-thick}.
\end{dfs}

Note that in the above definition it is actually enough to require
that the property holds for a sufficiently generic projection
$\pi: R^n \to R^j$. As we said in the introduction, roughly
speaking, $\N T^2$-thin sets of $k(0_+)^n$ are the generic fibers
of one parameter families whose fibers ``collapse onto a lower
dimensional subset at rate at least $t^2$" (if $t$ is the
parameter of the family). Also, by convention $R^0=\{0\}$ so that
a $0$-dimensional subset is never $v$-thin. This
is natural in the sense that  a family of  points never collapses onto a lower dimensional subset. 

\medskip

\noindent{\bf Basic properties of $(j;v)$-thin sets.} 
\noindent $(1)$ If a definable subset $A \subset X$ is
$(j;v)$-thin and if $h:X \to Y$ is a definable
 Lipschitz map then $h(A)$ is $(j;v)$-thin.

\noindent $(2)$ Given $j$, $\cup_{i=1}^p X_i$ is $(j;v)$-thin iff
$X_i$ is $(j;v)$-thin for any $i=1,\dots,p$.

\subsection{Definition of the vanishing homology.}
Given a definable set $X$ let $C^v_{j}(X)$ be the $G$-submodule of
$C_j(X)$ generated by all the singular chains $c$ such that $|c|$
is $(j;v)$-thin   and  $|\partial c|$ is $(j;v)$-thin as well. We
endow this complex with the usual boundary operator and denote by
$Z_j ^v(X)$ the cycles of $C_j ^v(X)$.

%

A chain $\sigma \in C_j^v(X)$ is said {\bf $v$-admissible}. 
 We denote by $H_j ^v(X)$ the resulting homology
groups which we  call the {\bf $v$-vanishing homology groups}.

If $v$ is $\N z$, for some $z \in R$ (see Remark \ref{rk velocity
et points}), then we will simply write $C_j ^z (X)$ and $H_j ^z
(X)$ (rather than $C^{\N z} _j$ and $H_j ^{\N z})$.

\begin{rk}\label{rmk x vthin vhom is hom}
If $X$ is $v$-thin and if $j=\dim X$ then every $j$-chain is
$v$-admissible. Moreover every $(j+1)$-dimensional chain is
admissible by definition. Hence the map $H_j^v(X) \to H_j(X)$
induced by the inclusion of the chain complexes is an isomorphism.
Note also that the   map $H_{j-1} ^v(X) \to H_{j-1} (X)$ is a
monomorphism.
\end{rk}
\medskip

Every Lipschitz map sends a $(j;v)$-thin set onto a $(j;v)$-thin
set. Thus, every Lipschitz map $f:X \to Y$, where $X$ and $Y$ are
two definable subsets, induces a sequence of mappings $f_{j,v}:H_j
^v (X) \to  H_j ^v (Y)$. In consequence, the vanishing homology
groups are preserved by definable bi-Lipschitz homeomorphisms.

\medskip


\medskip
As we said in the introduction this homology gives rise to a
metric  invariant for families (preserved by families of
bi-Lipschitz   homeomorphisms) by considering the generic fiber as
described in the following example.

\begin{ex}\label{ex fam} With the notation of example \ref{ex 1}, given an algebraic
family $X \subset \R ^n \times \R$ defined by $f_1=\dots=f_p =0$,
 we set
 $$X_{0_+}:=\{x\in k(0_+)^n: f_1(x;T) =\dots=f_{p}(x;T)=0\}.$$
 Hence, $H_j ^v (X_{0_+})$ is a metric invariant of the family.
\end{ex}

\subsection{The complex $ C_j ^v (X;\F)$.}  Given a finite family
$\F$, of closed subsets of $X$, we write $C_j(X;\F)$ for the
$j$-chains of $\underset{F \in\F}{\oplus }C_j(F)$. Similarly we
set:  $$C_j ^v (X;\F):= \underset{F \in \F}{\oplus} C_j ^v (F)$$
and denote by $H_j ^v (X;\F)$ the corresponding homology groups.
By Remark \ref{rmk x vthin vhom is hom}, if $\tau$ is a chain of
$Z_j ^v (|\sigma|)$ whose class is $\sigma$ in $H_j (|\sigma|)$
then $\tau=\sigma$ in $H_j ^v (|\sigma|)$ as well. Therefore, as
$H_j (|\sigma|;\F)=H_j (|\sigma|)$ we get:
\begin{equation}\label{eq subdivision hom} H_j ^v(X;\F)\simeq H_j ^v
(X).
\end{equation}

\subsection{Strongly admissible chains.}It is  difficult to construct  homotopies between
 $v$-admissible chains. To overcome this difficulty we introduce
 strongly $v$-admissible chains.

\begin{df}\label{def strongly admissible chains}We denote by $T_j^q$ the set of all $(x;\lambda) \in
T_j \times R$ such that $x +\lambda e_q$ belongs to $T_j$. A
simplex $\sigma:T_j \to R$ is {\bf strongly $v$-admissible} if
there exists $q$ such that for any $(x;\lambda) \in T_j^q$:
\begin{equation}\label{eq def fortement admissible}(\sigma(x)-\sigma(x +\lambda e_q)) \in v.\end{equation}
\end{df}

A chain is strongly admissible if it is a combination of strongly
admissible simplices. We  denote by $\widehat{C}_j ^v(X)$ the
chain complex generated by the strongly admissible chains $\sigma$
for which $\partial \sigma$ is strongly admissible, and by
$\widehat{Z}_j (X)$ the strongly admissible cycles. The resulting
homology is denoted by $\widehat{H}_j ^v (X)$. If $\F$ is a family
of closed subsets of $X$, we also define $\widehat{C}_j ^v
(X;\F)$, $\widehat{Z}_j ^v (X;\F)$, and $\widehat{H}_j ^v (X;\F)$
in an analogous way (see section $1.4$).

\begin{rk}\label{rk strongly admissible is admissible}
Let $\sigma :T_j \to R^n$ be a strongly admissible simplex with $j
\leq n$. Then by definition, there exists $z \in v$ such that for
any $x \in T_j$:
$$d(\sigma(x);\sigma(\partial T_j)) \leq z.$$

As $\sigma(\partial T_j)$ is of dimension strictly inferior to $j$
  we see that the image of this set under a  projection onto
$R^j$ contains no open ball in $R^j$. In other words, if $\sigma$
and $\partial \sigma$ are strongly admissible chains then $\sigma$
is admissible. In consequence, a strongly admissible cycle is
admissible.
\end{rk}

\section{Rectilinearizations of $v$-thin sets.}
\subsection{Regular directions.}
We recall a result proved in \cite{v1} which will be very useful
to compute our vanishing homology. We start by the definition of a
regular direction. We denote by $X_{reg}$ the set of points $x \in
X$ at which $X$ is a $C^{1}$ manifold.
\begin{df}\label{def reg proj}
Let $X$ be a definable set of $R^{n}$. An element $\lambda$ of
$S^{n-1} $ is said  {\bf regular  for $X$} if there exists
 a positive $\alpha \in \Q$:
$$d(\lambda;T_x X_{reg}) \geq \alpha,$$
for any $x \in X_{reg}$.
\end{df}

Not every definable set has a regular line. However, we have:
\begin{prop}\label{proj reg hom pres}\cite{v1}
Let $A$ be a definable subset of $R^n$ of empty interior. Then
there exists a definable  bi-Lipschitz homeomorphism $h : R^n
\rightarrow R^n$ such that $e_n$ is regular for  $h(A)$.
\end{prop}

\begin{rk}\label{rk globally defined xi}
When $e_n$ is regular for a set $X$, we may find finitely many
Lipschitz definable functions, say $\xi_i:R^{n-1} \to R$,
$i=1,\dots,s,$ satisfying \begin{equation}\label{eq xi ordonnes
rk}\xi_1 \leq \dots \leq \xi_s,\end{equation} and such that the
set $X$ is included in the union of  their respective graphs.
\end{rk}

\medskip

\subsection{Cell decompositions.}
 In order to fix notations we recall
the definition of the cells, which, as usual, are introduced
inductively. All the definitions of this section deal with subsets
of $R^n$, but since $R$ stands for an arbitrary real closed field,
we will use them for subsets of $k_v^n$ as well.

\begin{dfs}\label{def thin cells}
For $n=0$ a cell of $R^n$ is $\{0\}$. A {\bf cell} $E$ of $R^n$ is
either the graph of a definable function $\xi:E' \to R$, where
$E'$ is a cell of $R^{n-1}$ or a band  of type:
\begin{equation}\label{eq def cell}\{x=(x';x_n)\in E'\times R: \xi_1(x') < x_n< \xi_2(x')\},\end{equation}
where $\xi_{1}, \xi_2:E' \to R$ are two definable functions
satisfying $\xi_1< \xi_2$ or $\pm \infty$. The cell $E$ is {\bf
Lipschitz} if $E'$ is Lipschitz and if  $\xi_1$ and $\xi_2$ (or
$\xi$) are Lipschitz functions (and $\{0\}$ is Lipschitz). A {\bf
closed cell} is the closure of a cell (which is obtained by
replacing $<$ by $\leq$ in the definition).


 Given $z \in R$, the Lipschitz cell $E$ is {\bf $z$-admissible} if\begin{enumerate} \item  $E'$ is
$z$-admissible  \item If $E$ is a band  defined by two functions
$\xi_1$ and $\xi_2$, then either  $(\xi_2-\xi_1) (x) \leq  z$ for
 any $x \in E'$, or  $(\xi_2-\xi_1)(x) \geq z$ for
 any $x \in E'$.
\end{enumerate}

Set also that the cell $\{0\}$ is $z$-admissible.


A cell $E$ of dimension $j$ is canonically homeomorphic to
$(0;1)^j$. The {\bf barycentric subdivision of $E$} is the
partition defined by the image  by this homeomorphism of the
barycentric subdivision of $(0;1)^j$.
\end{dfs}

We shall need the following very easy lemma.

\begin{lem}\label{lem cell thick}
Let $E'$ be a $w$-thick Lipschitz cell of $k_v^{n-1}$ and let
$\xi_{s}:E' \to k_v$, $s=1,2$, be two Lipschitz functions such
that $\xi_1 <\xi_2$ and $(\xi_1-\xi_2 )(x) \notin w$, for any
$x\in E'$. Then the band:
$$E:=\{(x;y) \in E' \times k_v:\xi_1(x)< y <\xi_2(x) \}$$ is $w$-thick.
\end{lem}
\begin{proof}
We may assume that $E'$ is  open in $k_v^{n-1}$ since we may find
a bi-Lipschitz homeomorphism which carries $E'$ onto an open cell.
Then $E$ is also an open and the cell $E'$ contains a ball of
radius $z \notin w$, say $B(x_0;z)$. Let
$t:=\xi_2(x_0)-\xi_1(x_0)$; we have by assumption $t \notin w$.
Taking $t$ small enough we may assume $B(x_0;t)$ entirely lies in
$E'$. Let $N$ be the Lipschitz constant of $(\xi_1-\xi_2)$ and
note that: $$\xi_2(x)-\xi_1(x) \geq \frac{t}{2},$$ for $x \in
B(x_0;\frac{t}{2N})$. This implies that $E$ contains a ball of
radius $\frac{t}{2N}$. But, as $t \notin w$ we have $\frac{t}{2N}
\notin w$.
\end{proof}
\medskip

\begin{df}\label{def L cell dec}The subset  $\{0\}$ is an
{\bf $L$-cell decomposition of $R^0$}. For $n >0$, an {\bf
$L$-cell decomposition} of $R^n$ is a cell decomposition of $R^n$
satisfying:
\begin{enumerate} \item[(i)]The cells of $R^{n-1}$ constitute an
$L$-cell decomposition of $R^{n-1}$
 \item[(ii)]   There exist finitely many Lipschitz functions
  $\xi_1,\dots,\xi_s:R^{n-1}\to R$ satisfying (\ref{eq xi ordonnes  rk})
   such that the union of all the  cells which are graphs of a function on a subset of $R^{n-1}$,
    is  the union of the  graphs of the $\xi_i$'s. \end{enumerate}
 An $L$-cell decomposition is
said {\bf compatible with} finitely many given definable subsets
$X_1,\dots,X_m$ if these subsets are union of cells. It is said
{\bf $z$-admissible} if every cell is $z$-admissible. Taking the
barycentric subdivision of every cell, we get a {\bf barycentric
subdivision} of an $L$-cell decomposition.
\end{df}

\medskip

   We are going to show  that,   we may find a $u$-admissible  $L$-cell
decomposition which is compatible with some given
$\la_R(u)$-definable subsets of $k_v^n$. This will be helpful to
prove that the homology groups are finitely generated, since we
will show that only the $\N u$-thin cells are relevant to compute
the homology groups.  The following proposition deals with subsets
of $k_v$ since we will apply it to $k_v$ but of course the proof
goes over an arbitrary model of the theory.

\begin{prop}\label{prop cell dec}
Let  $X_1,\dots,X_m$ be   $\la_R(u)$-definable  subsets of $k^n
_v$. There exists a $\la_R (u)$ definable bi-Lipschitz
homeomorphism $h:k^n _v\to k^n _v$ such that we can find a
$u$-admissible
$L$-cell decomposition of $k_v^n$ compatible with $h(X_1),\dots,h(X_m)$. 
\end{prop}
\begin{proof}
 For $n=0$ there is nothing to
prove. Assume $n
>1$ and apply Proposition \ref{proj reg hom pres} to $\cup_{j=1} ^m \partial X_j$ (where $\partial $ denotes the topological boundary). Then
(see Remark \ref{rk globally defined xi}) there exist finitely
many definable Lipschitz functions $\xi_i$, $i=1,\dots ,s$
satisfying (\ref{eq xi ordonnes rk}). Consider a cell
decomposition of $k_v^{n}$ compatible with $X_1,\dots,X_m$, all
the graphs of the $\xi_i$'s, as well as all the sets $$\{x\in
k_v^{n-1} :\xi_{i+1}(x) - \xi_i(x) =u\}.$$ Now apply the induction
hypothesis to all the cells of this decomposition which lie in
$k_v^{n-1}$ to get a cell decomposition $\mathcal{E}$ of
$k_v^{n-1}$. Then set $\xi_0:=-\infty$, $\xi_{s+1}:=\infty$, and
consider the cell decomposition of $k_v^n$ constituted by the
graphs of the restrictions of the functions $\xi_i$'s to an
element of $\mathcal{E}$  on the one hand, and all the subsets of
type:
$$\{(x;x_n) \in E \times k_v : \xi_i(x) < x_n  < \xi_{i+1}(x)\},$$
where $E \in \mathcal{E}$, on the other hand. The required
properties hold.
\end{proof}

\bigskip

%
%
\subsection{Rectilinearization of $v$-thin sets.}
We introduce the notion of rectilinearization. This is a mapping
which transforms a set into a union of coordinate hyperplanes and
which induces an isomorphism in homology (the usual one).
Admissible rectilinearizations will be very helpful to construct
strongly admissible chains (see section $1.5$).
 We are going to  show that  we can always find a $v$-admissible rectilinearization  compatible
with a given family of $v$-thin sets.

\begin{dfs}
A {\bf hyperplane complex} is a subset $W$ of $R^n$, which is a
union of finitely many coordinate hyperplanes of type $x_j=s$
where, for each hyperplane, $s$ is an integer. There is a
canonical cell decomposition of $R^n$ compatible with $W$. We
refer to the cells (resp. closure of the cells) as the  {\bf cells
of $W$} (resp. {\bf closed cells of $W$}).

Let $X_1,\dots,X_m$ be definable  subsets.  A {\bf
rectilinearization} of $X_1,\dots,X_m$
 is a  mapping $h:R^n \to R^n$, such that the  $h^{-1}(X_i)$'s are union of cells of $W$ and such that   for any $i=1,\dots,m$ the
mapping $h_i:h^{-1}(X_i) \to X_i$ induces an isomorphism in
homology (the usual one).

 If  $X_1,\dots,X_m$ are $v$-thin,  a   rectilinearization of $X_1,\dots,X_m$ is {\bf $v$-admissible} if
for each  cell $\sigma$ of $W$ included in $h^{-1}(X_i)$  there
exists an integer $q$ with $e_q$ tangent to $\sigma$ for which
\begin{equation}\label{eq def de rect}(h(x)-h(x +\lambda e_q))
\in  v\end{equation} for any $x \in \sigma$ and $\lambda \in R$
such that
$x+\lambda e_q \in \sigma$. 
\end{dfs}

\begin{rk}\label{rk bar subd rect}
After a barycentric subdivision of $h^{-1}(X_i)$, we get a
simplicial complex $K_i$ and a map $h_i:K_i \to X_i$ which induces
an isomorphism in homology. Note that, thanks to (\ref{eq def de
rect}) each simplicial  chain gives rise  (identifying each
$j$-simplex to $T_j$ in a linear way) to a strongly admissible
chain (see Definition \ref{def strongly admissible chains}).
Moreover, as $h$ induces an isomorphism in homology, this
identification defines an isomorphism in homology $H_j (K_i) \to
H_j (X_i)$.
\end{rk}

\medskip

%



\begin{prop}\label{prop caract v thin sets}
Let $X_1,\dots, X_m$ be  closed definable $v$-thin subsets of
$R^n$. Then there exists a   $v$-admissible rectilinearization of
$X_1,\dots, X_m$.
  \end{prop}
\begin{proof}
We start by  proving  the following statements $\mathbf{(H_n)}$ by
induction on $n$.

\medskip

 \noindent   $\mathbf{(H_n).\;}$
 Let $\mathcal{ E}$ be a $u$-admissible  $L$-cell
decomposition of $k_v^n$ and let $Y_1,\dots,Y_r$ denote the
$w$-thin closed cells.
 Then there exists   a $\N u$-admissible
    rectilinearization $h : k_v^n \to k_v^n$  of  $Y_1,\dots,Y_r$ such
   that,
   for every  $E$ in $\mathcal{E}$, $h^{-1}(cl(E))$ is a union of closed cells of $W$ and  there exists a strong
 deformation retract $r_E :h^{-1}(cl(E)) \to C_E$, where $C_E$ is a closed cell of $W$.
\medskip

Note that it follows from the existence of this deformation
retract that $h$ induces an isomorphism in homology above any
union of closed   cells of $\mathcal{E}$. Actually, the existence
of $r_{Y_i}$ implies
$$H_j (h^{-1}(Y_i))\simeq H_j (C_{Y_i}) \simeq H_j (Y_i), $$
and the map $h_{|h^{-1}(Y_i)}:h^{-1}(Y_i) \to Y_i$ induces an
 isomorphism in homology. Therefore,  thanks to the
 Mayer-Vietoris property and to the $5$-Lemma, we see that  for any subset $X$ constituted by the union of finitely many closed  cells  the map
 $h_{|h^{-1}(X)}:h^{-1}(X) \to X$ induces an isomorphism in homology.
\medskip

Note that nothing is to be proved for $n=0$ and assume
$\mathbf{(H_{n-1})}$. Apply the induction hypothesis to the family
constituted by  the closure of the  cells of $\mathcal{E}$ in
$k_v^{n-1}$ which are $w$-thin to get a rectilinearization $h:k_v
^{n-1} \to k_v ^{n-1}$ and a hyperplane complex $W$.

%


 Note that by definition, the
cells of $\mathcal{E}$ on which the restriction of $\pi_n$ is
one-to-one are included in the union of finitely many graphs of
definable Lipschitz functions $\xi_1,\dots,\xi_s: k_v ^{n-1} \to
k_v$ satisfying (\ref{eq xi ordonnes rk}).

We obtain a hyperplane complex  $\widetilde{W}$ by taking the
inverse image of $W$ by $\pi_n$, and by adding the hyperplanes
defined by $x_n=i$,
$i=1,\dots ,s$.

Define now  the desired mapping $\widetilde{h}$ as follows:
$$\widetilde{h}(x; i+t)=
(h(x);\,(1-t) \xi_i(h(x))+t\xi_{i+1}(h(x)))$$
 for $ 1 \leq i < s$ integer, $x \in k_v^n$ and $t \in [0;1)$.
Define also:
$$\widetilde{h}(x;\, 1-t)=
(h(x);\xi_1(h(x))-t)$$ and
$$\widetilde{h}(x; \, s+t)=
(h(x);\xi_s(h(x))+t)$$
  for  $t \in [0;\infty)$. This defines a mapping  $\widetilde{h}:
  k_v^n\rightarrow k_v^{n}$. We are going to check that
  \begin{equation}\label{eq h tilda est u admissible}
|\widetilde{h}(x)-\widetilde{h}(x +\lambda e_n)| \leq
 u\end{equation} when $x$ and $(x+\lambda e_n)$ belong to
 the same cell.

Let $\sigma$ be a cell of $\widetilde{W}$ which is mapped into
$\cup_{i=1} ^r Y_i$.  If $\pi_n (\sigma)$ is $w$-thin (\ref{eq h
tilda est u admissible}) follows from the induction hypothesis.
Otherwise $\widetilde{h}(\sigma)$ must lie in the band delimited
by the graphs of the restrictions of $\xi_i$ and $\xi_{i+1}$ for
some $i \in \{1,\dots,s-1\}$ as described in (\ref{eq def cell}).
If $\widetilde{h}(\pi_n(\sigma))$ fails to be $w$-thin then,
thanks to Lemma \ref{lem cell thick} (recall that
$\widetilde{h}(\sigma)$ is $w$-thin) and the $u$-admissibility of
the cell decomposition, we necessarily have:
$$|\xi_i(x) -\xi_{i+1}(x)| \leq u,$$ for any $x \in \pi_n(\sigma)$.
This, together with definition of $\widetilde{h}$, implies that
$\widetilde{h}$ satisfies
 (\ref{eq h tilda est u admissible}) and yields that
$\widetilde{h}$ is $ \N u$-admissible. It remains to find the
retraction $r_E$ for each  cell $E$.

Fix $E \in \mathcal{E}$ and observe that it follows from the
definition of $\widetilde{h}$ and  the induction hypothesis that
$\widetilde{h}^{-1}(cl(E))$ is a union of cells of
$\widetilde{W}$. If $E$ is the graph of a function $\xi:E' \to
k_v$ (where  $E':=\pi_n(E)$), then the result directly follows
from the induction hypothesis. Otherwise, since $\mathcal{E}$ is
an $L$-cell decomposition, the cell $E$ lies in the band delimited
by the graphs of two consecutive functions, say $\xi_i$ and
$\xi_{+1}$.
 Let
$$\Gamma_i:= \{(x;x_n) \in k_v ^{n-1} \times k_v :i\leq x_n \leq i+1
\}. $$

We first define first a retract:
$$r_{E}': \widetilde{h}^{-1}(cl(E)) \times [0;\frac{1}{2}]_{k_v} \to \Gamma_i \cap \widetilde{h}^{-1}(cl(E)) , $$
by setting for $x_n \geq i+1$:$$r'_E(x;x_n;t):=(x;2t x_n+
(1-2t)(i+1)),$$  and for $x_n \leq i$:$$r'_E(x;x_n;t):=(x;2t x_n+
(1-2t)i),$$ and of course $r_E'(x;x_n;t):=(x;x_n)$ when $i \leq
x_n \leq i+1$.

 Note that it follows from the
definition of $\widetilde{h}$ that if $(x;x_n)$ belongs to
$\widetilde{h}^{-1}(cl(E))$ then for any $i+1 \leq x_n '\leq x_n$
and any $x_n \leq x_n '\leq i$:
$$\widetilde{h}(x;x'_n)=\widetilde{h}(x;x_n).$$ This implies that $r'_E$ preserves
$\widetilde{h}^{-1}(cl(E))$.

 On the other hand, thanks to the induction hypothesis, there exists a
retract  $r_{E'}: h^{-1}(cl(E')) \times [0;1]_{k_v} \to C_{E'}$.
Let us extend this $r_{E'}$ into a retract:
 $$r''_{E'}: \pi_n^{-1}(h^{-1}(cl(E'))) \times [\frac{1}{2};1]_{k_v} \to \pi_n ^{-1}(C_{E'})$$
by $$r'_E(x;x_n;t):=(r_{E'}(x;2t-1);x_n).$$

Clearly, there exists a unique cell $C_E$ of $\widetilde{W}$ which
is included in $\Gamma_i$ and which projects on $C_{E'}$. Now,
these retracts give rise to a retract
$$\widetilde{r}_E :\widetilde{h}^{-1}(cl(E)) \times [0;1]_{k_v} \to
C_E$$ defined by $\widetilde{r}_E(x;t):=r'_E(x;t)$ if $t \leq
\frac{1}{2}$ and
$$\widetilde{r}_E(x;t):=r''_E(r'_E(x;\frac{1}{2});t)$$
 if $t \geq \frac{1}{2}$. This yields $\mathbf{(H_n)}$.

\bigskip

We return to the proof of the proposition. Apply Proposition
\ref{prop cell dec} to $X_{1,v},\dots,X_{m,v}$. This provides a
bi-Lipschitz  homeomorphism $g: k_v^n \to k_v^n$ such that we can
find a $u$-admissible $L$-cell decomposition of $k_v^n$ compatible
with $g(X_{1,v}),\dots,g(X_{m,v})$. Note that, as the
$g(X_{i,v})$'s are $w$-thin, each of them is the   union of some
$w$-thin cells. Then by $\mathbf{(H_n)}$, there exists a $\N
u$-admissible rectilinearization of these cells $h:k_v^n \to
k_v^n$.

Composing with $g$, the mapping $h$  gives rise to a $\N
u$-admissible rectilinearization $f$ of
  $X_{1,v},\dots,X_{m,v}$. As the  $X_{i,v}$ are extensions,  there exist two families of
rectilinearizations $f_z$ and $h_z$ for $z \in [a;b]$ with $a < u
<b$ and $a,b\in R$. Let us check that these rectilinearizations
are $v$-admissible for $z \in v$ large enough.



Note that each $X_i$  is the  union of the  images by $h_z$ of
finitely many cells of $W$.  Furthermore, as (\ref{eq h tilda est
u admissible}) is a first order formula we get that $h_z$
satisfies on any given cell in the inverse image of  the $X_i$'s:
$$|h_z(x)-h_z(x +\lambda
e_n) |\leq  z, $$ when $x$ and $(x+\lambda e_n)$ belong to this
given cell.

 This implies that $f$ satisfies (since $g$ is bi-Lipschitz):
$$|f_z(x)-f_z(x +\lambda
e_n) |\leq N z, $$ for some $N \in \N$ and any $z \in v$ large
enough on any cell mapped into one of the $X_i$'s. Thus,  (\ref{eq
def de rect}) holds and  $f_z$ is $w$-admissible.
  \end{proof}

\begin{rk} Actually, working a little more, we could have proved
that the constructed rectilinearization induces an  isomorphism in
homology above any subset $A$ of $R^n$. Namely, in the above
proof, given a subset $A$ of $R^n$, the induced mapping
$\widetilde{h}:\widetilde{h}^{-1}(A) \to A$ induces an isomorphism
in  homology.

Observe also that the constructed mapping is a homeomorphism above
a dense definable subset. If we take an algebraic hypersurface,
the situation is fairly similar to the one which occurs with
resolution of singularities in the sense that the inverse image of
the set above which the map is not one-to-one (the ``exceptional
divisor") is constituted by finitely many coordinate
  hyperplanes
 normal to the hyperplanes lying above our given set. We could also
 have a more precise description of how the mapping $h$ modifies the distances (like in \cite{v1}). More
 precisely, it is possible to see that on each cell,
 we have $$|h(x)-h(x')| \sim  \sum _{i=1} ^n \varphi_i (x)|x_i -x'_i|$$
  where $\varphi_i$ is a sum of product of powers of
distances to cells of $W$. If we compare this result with Theorem
$5.1.3$ of \cite{v1}, we see that now the contractions (see
\cite{v1}) are expressed in the canonical basis. The inconvenient
is that the map $h$ is not a homeomorphism (contrarily as in
\cite{v1}), but since it induces  an isomorphism between the
homology groups, it will be enough for the purpose of the present
paper.
\end{rk}




%
%
%
%
%
%
%

\section{The vanishing homology groups are finitely generated}
\subsection{Some preliminary lemmas.}
%
Every mapping $\sigma: T_j \to X$ may  be extended to a mapping
$\sigma_v : T_j (k_v)\to X_v$ (see subsection $1.1$). Let
$\Delta^{ext} _j (X_v) $ be the submodule of $C_j ^w(X_v)$
generated by the simplices which are extensions of an element of
$C^v _j (X)$. Clearly, for each $j$ the mapping:
$$ext :C_j ^v (X) \to \Delta_j ^{ext}(X_v),$$
which assigns to every chain $\sigma$ the chain $\sigma_v$,
induces an isomorphism in homology.


\medskip

 The following Lemma says that the
vanishing homology groups for the velocities $\N u$ and  $w$
coincide with the homology groups of $\Delta^{ext}_j$ when the
considered set is $\la_R$-definable.

\begin{lem}\label{lem w, ext  et nu}
Let $X$  be a definable subset of $R^n$. Then the maps induced by
the inclusions $H_j(\Delta^{ext} (X_v)) \to H^{  u}_j(X_v)$ and
$H_j(\Delta^{ext} (X_v)) \to H^w_j(X_v)$   are isomorphisms for
any $j$.
\end{lem}
\begin{proof}
We do the proof for $u$. To get the proof for $w$, just replace $
u$ by $w$. We first check that this map is onto. Let
$\sigma=\sum_{i\in I}  g_i c_i \in Z^{ u}_j(X_v)$. By definition
of $u$ there exist finitely many $\la_R$-definable mappings, say
$\tau_i:T_j(k_v) \times [a;u]_{k_v} \to X_v$, with $a \in v$  such
that $c_i(x)=\tau_{i}(x;u)$
  for any $x \in T_j(k_v)$. Define
$\theta_i(x):=\tau_{i}(x;a)$ and $\theta:=\sum_{i \in I} g_i
\theta_i \in C_j ^{ext}(X_v)$. Observe that $\tau_{i}$ gives rise
to a $\N u$-admissible $(j+1)$-chain (after a subdivision of
$T_j(k_v) \times [a;u]$). Moreover, as the property of
admissibility may be expressed by a formula with parameters in $R$
and with $u$, we know that the obtained chain is $\N u$-admissible
if $a$ is chosen large enough. Set $\tau :=\sum_{i \in I} g_i
\tau_i \in C_{j+1} ^{ u}(X_v)$ and note that since
$\tau_i(x;u)=c_i(x)$ and $\tau_i(x;a)=\theta_i(x)$ we clearly have
$\partial \tau=\sigma-\theta$. As $\theta$ belongs to
$C^{ext}_j(X_v)$, this implies that the inclusion  $C^{ext}_j(X_v)
\to C^{ u}_j(X_v) $ induces a surjection in homology.

We now check that this map is injective by applying a similar
argument. Let $\alpha \in C^{ext}_j(X_v)$ with $\alpha =\partial
\sigma$ where $\sigma$ belongs to $C^{  u} _{j+1}(X_v) $. The
chain $\sigma$
 induces  chains $\tau \in C_{j+2} ^{ u}(X_v)$ and $\theta \in C_{j+1} ^{ext} (X_v)$ such that  $\partial \tau= \sigma-\theta
$   in the same way as in the previous paragraph. But this implies
$\partial \theta=\alpha$ which means that $\alpha \in
\partial  C_{j+1} ^{ext} (X_v)$, as required.
\end{proof}

Given a definable family $Y$ of $R^n \times R$ and $t \in R$, we
denote by  $Y_t$ the {\bf fiber at $t$}: $$\{x \in R^n: (x;t) \in
Y\}.$$ 

 We also
define the {\bf restriction of the family} to $[a;b]$ as follows:
$$Y_{[a;b]}:=\{(x;t) \in Y: a \leq t\leq b\}.$$
\medskip

\begin{lem}\label{lem N u egal N t}
Let $Y$ be a $\la_R(u)$-definable family of $k_v^n \times k_v$
such that $Y_u$ is a $\N u$-thin subset of $k_v^n$ and let $j=\dim
Y_u$. Then there exists  $z$ in $v$  such that for any $t \in v$
greater than $z$ the map induced by inclusion:
$$H_k ^{ w}(Y_t) \to H_k ^{u}(Y_{[z;u]}),$$
 is an isomorphism  for $k=j$ and  is  one-to-one  for $k=j-1$. 
\end{lem}
\begin{proof}
As $Y$ is $\la_R(u)$-definable and $\N u$-thin there exists $z$ in
$v$ such that for any $t$ in $v$ greater than $z$, $Y_t$ is
$w$-thin. Thanks to Remark \ref{rmk x vthin vhom is hom}, this
implies that the natural mapping  $H_j ^{ w}(Y_t)\to H_j(Y_t)$ is
an isomorphism.

 Furthermore, since the family $Y$ is topologically trivial if
the interval $[z;u]$ is chosen small, the inclusion $H_j (Y_t) \to
H_j (Y_{[z;u]})$ induces an isomorphism in homology as well.

We have the following commutative diagram for $t \in v$ greater
than $z$:
\begin{center}
     \begin{picture}(-140,0)
\put(-120,3){$1$} \put(-120,-47){$4$}
\put(-160,-25){$3$}\put(-65,-25){$2$}
      \put(-180,-3){$H_j ^w(Y_t)$}
        \put(-90,-3){$H_j (Y_t)$}
      \put(-190,-52){$H_j ^u(Y_{[z;u]})$}
          \put(-90,-52){$H_j (Y_{[z;u]})$}
      \put(-140,-50){\vector(3,0){40}}
      \put(-140,0){\vector(3,0){40}}
  \put(-165,-10){\vector(0,-1){30}}
      \put(-70,-10){\vector(0,-1){30}}
     \end{picture}
    \end{center}
\vskip 20mm

By the above, the  arrows $1$ and $2$ are isomorphisms. Moreover
as $Y_u$ is $\N u$-thin the family $Y_{[z;u]}$ is $\N u$-thin.
Thus, the arrow $4$ is an monomorphism (see the last sentence of
Remark \ref{rmk x vthin vhom is hom}). This implies that the arrow
$3$ is an isomorphism and establishes the theorem in the case
$k=j$.

Now, in the case where $k=j-1$ we can write the same diagram for
$H_{j-1}$. The arrows $1$ and $2$ (of the obtained diagram) are
still one-to-one (again thanks to Remark \ref{rmk x vthin vhom is
hom} and the topological triviality of $Y_{[z;u]}$), so that the
arrow $3$ is clearly one-to-one.
\end{proof}


\bigskip

%

The following lemma is a consequence of existence of
$v$-admissible rectilinearizations.
\begin{lem}\label{lem strong adm is onto}
Given $X \subset k_v^n$  $\la_R(u)$-definable and $\F$ finite
family
 of closed $\la_R(u)$-definable subsets of $X$, the map
$\widehat{H}_j^w(X;\F) \to H_j ^w(X)$, induced by the inclusion,
is onto.
  \end{lem}
\begin{proof}
Let $\sigma \in C_j^{w}(X)$. If the support of $\sigma$ is of
dimension $<j$ then the class of $\sigma$ is $0$ in $H_j ^w(X)$.
Thus, we may assume that $dim |\sigma|=j$.

Let $h:k_v ^n \to k_v ^n$ be a
$w$-admissible rectilinearization of  $|\sigma|$ and of all the elements of $\F$. 
There exists a simplicial chain $\tau$ (see Remark \ref{rk bar
subd rect}), which is strongly $w$-admissible since $h$ is
$w$-admissible, such that $\sigma= \tau$ in
$H_j(|\sigma|)=H_j^w(|\sigma|)$ (see Remark \ref{rmk x vthin vhom
is hom}). But this means that the class of $\tau$ is that of
$\sigma$ also in $H_j^w(X)$. This yields that the inclusion
induces an onto map in homology.
\end{proof}

It is unclear for the author whether the inclusion of the above
lemma is one-to-one. Actually, it is even unclear whether
$\widehat{H}^v_j(X)$ is finitely generated.

\medskip


\subsection{The main result.} It is very hard to construct
homotopies which are Lipschitz mappings. To compute the homology,
we actually just need to find a homotopy that carries a chain
$\sigma$  to the cells of a given cell decomposition, and which
preserves the $v$-admissibility of the chain $\sigma$. We prove
something even weaker: given a strongly $w$-admissible chain, we
may construct a homotopy which carries the chain $\sigma$ to a
strongly $\N u$-admissible chain of the cells of dimension $j$.
This is enough since we have seen that we had  isomorphisms
between the theories defined by $w$ and $\N u$. This technical
step is performed in the following proposition.

\begin{prop}\label{prop  homotopie} Let $X$ be a closed $\la_R (u)$-definable subset of $k_v^n$ and
let $\mathcal{E}$ be a $u$-admissible $L$-cell decomposition
compatible with $X$. Let $\F$ be the family  constituted by  the
closed    cells of $\mathcal{E}$ and let  $Y_j$ be the union of
the closures of the $(\N u;j)$-thin elements of  the barycentric
subdivision of $\mathcal{E}$. Then, there exists a map
$$\varphi:
 \widehat{C} _j ^w (X;\F) \to\widehat{C}_j ^u (Y_j)$$ such that:
\begin{enumerate}
\item[$(i)$] $\quad \varphi \partial -\partial \varphi =0$
\item[$(ii)$] For any $\sigma \in \widehat{Z}_j ^w(X;\F)$ we have:
$\varphi_\sigma=\sigma,$ in $H_j ^u (X)$, \item[$(iii)$] If $Y$ is
the union of some elements of $\F$,  then for any $\sigma \in
\widehat{Z}_j ^w(X;\F)$ with $|\sigma| \subset Y$ we have:
$\varphi_\sigma=\sigma$ in $ H_j ^u (Y)$.
\end{enumerate}
\end{prop}

 \begin{proof}
 We are going to prove the following statements:

\noindent {\bf Claim.} Given $\sigma \in C_j (X;\F)$, there exists
a definable  homotopy $$h_\sigma: T_j(k_v) \times [0;1]_{k_v} \to
X,
$$
such that: \begin{enumerate}\item For each $x$ the path $t \mapsto
h_\sigma(x;t)$ stays in the same closed cell, \item For each $t$
the map $x \mapsto h_\sigma(x;t)$  is a strongly $\N u$-admissible
simplex if $\sigma$ is a strongly $w$-admissible simples, \item If
$\sigma$ is strongly $w$-admissible, the support of the simplex
$\varphi _\sigma:T_j(k_v) \to X$ defined by
$\varphi_\sigma(x)=h_{\sigma}(x;1)$ entirely lies in $Y_j $ \item
 We have $$\partial h_* (\sigma)-h_* (\partial\sigma) =\varphi_\sigma
-\sigma$$ for any $\sigma \in C_j  (X;\F)$ where (as usual)
$h_*:C_j(X;\F) \to C_{j+1}(X;\F) $ is the mapping induced by $h$
on the chain complexes.

\end{enumerate}

Note that $\varphi$ is defined by $(3)$. Observe that $(4)$
implies $(i)$,   together with $(2)$ implies $(ii)$, and together
with $(1)$ yields $(iii)$.

We prove that it is possible to construct such a homotopy by
induction on $n$ (the dimension of the ambient space). Let
$\mathcal{E}'$ be the cell decomposition of $k_v^{n-1}$
constituted by all the cells of $\mathcal{E}$ lying in
$k_v^{n-1}$. Let $\sigma$ in $C_j (X;\F)$ and  write
$\sigma:=(\widetilde{\sigma};\sigma_n)\in k_v ^{n-1} \times k_v$.
 Apply the induction hypothesis to $\widetilde{\sigma}$ and $\mathcal{E}'$ to get a homotopy
$h_{\widetilde{\sigma}}: T_j(k_v) \times
[0;1]_{k_v}\to k_v^{n-1}$.  

By definition, the union of the  cells of $\mathcal{E}$ on which
$\pi_n$ is one-to-one is given by the graphs of finitely many
Lipschitz functions $\xi_1 \leq \dots \leq \xi_s$. Note that we
may retract the cells above (resp. below) the graph of $\xi_s$
(resp. $\xi_1$) onto the graph of $\xi_s$ (resp. $\xi_1$) so that
we may assume that $X$ entirely lies between these two graphs.

By compatibility with $\F$ we know that the support of $\sigma$
entirely lies in one single cell $E \in \mathcal{E}$  which is
either the graph of a Lipschitz function $\xi$ or a band which is
delimited by the graph of the restriction to $E':=\pi_n(E)$ of two
consecutive functions $\xi_i$ and $\xi_{i+1}$, with $\xi_i <
\xi_{i+1}$ on $E'$. In the latter case, we may define a function
$\nu_\sigma :T_j(k_v) \to [0;1]_{k_v}$ by setting for $x \in
T_j(k_v)$
$$\nu_\sigma(x):=\frac{\sigma_n(x)-\xi_i(\widetilde{\sigma}(x))}{\xi_{i+1}(\widetilde{\sigma}(x))-\xi_i(\widetilde{\sigma}(x))}.$$


 To deal with both cases simultaneously it is convenient to
set $\nu_\sigma (x)\equiv0$ and $\xi_i=\xi_{i+1}=\xi$, if the cell
is described by the graph of a single function $\xi$.  To define
$h_{\sigma}$ we first define a function $s_{\sigma}:T_j(k_v) \to
[0;1]_{k_v}$. We set:
\begin{eqnarray*} &s_\sigma(e_i)=0& \quad \mbox{if} \quad
\sigma_n(e_i)-\xi_i(\widetilde{\sigma}(e_i)) \in w \quad \,\,\; \mbox{and} \quad  \xi_{i+1}(\widetilde{\sigma}(e_i))-\sigma_n(e_i)\neq 0\\
\mbox{and} \:&s_\sigma(e_i)=1& \quad
\mbox{otherwise.}\end{eqnarray*} Then we extend $s_\sigma$ over
$T_j (k_v)$ linearly.

 Now we can set for $(x;t) \in
T_j(k_v) \times [0;\frac{1}{2}]_{k_v}$:
$$\theta(x;t)=2ts_\sigma(x)+(2t-1)\nu_\sigma (x).$$
 Set for simplicity: $\xi'=\xi_{i+1}-\xi_i$  and, for $x=(\widetilde{x};x_n)
\in k^{n-1} _v \times k_v$ and $t \in [0;1]_{k_v}$, let:
\begin{eqnarray*}&h_\sigma(x;t):=(\widetilde{\sigma}(x);\xi_i(\widetilde{\sigma}(x))+\theta(x;t)\xi'(\widetilde{\sigma}(x)))& \quad
\mbox{if} \quad t \leq \frac{1}{2}\\
&h_\sigma
(x;t):=(\,h_{\widetilde{\sigma}}(\widetilde{x};2t-1)\,;\,\xi_i(h_{\widetilde{\sigma}}(\widetilde{x};2t-1))
+s_\sigma
(x)\,\xi'(\,h_{\widetilde{\sigma}}(\widetilde{x};2t-1))\,)& \quad
\mbox{if} \quad t \geq \frac{1}{2}.
\end{eqnarray*}
Note that as $s_\sigma$ (resp. $\nu_{\sigma}$)  satisfies:
$$s_{\partial \sigma}=\partial s_\sigma $$ (resp. $\nu_{\partial \sigma}=\partial \nu_ \sigma
$), we see that the map induced by  $h_\sigma$ is a chain
homotopy. Moreover, it is clear from the definition of
$h_{\sigma}$ that the path $t \mapsto h_\sigma (x;t)$ remains in
the same closed cells. Therefore $(1)$ and $(4)$ hold.

To check $(2)$,  fix a strongly admissible simplex $\sigma$. We
have to check that there exists $q \in \{1,\dots,n\}$ such that:
\begin{equation}\label{eq preuve de (2)}(h_\sigma(x+\lambda
e_q;t)-h_\sigma(x;t)) \in \N u \end{equation} for any $(x;\lambda)
\in T_j^q(k_v)$ and any $t$ in $[0;1]_{k_v}$.

If $\sigma$ is the graph of one single function $\xi$ then the
result is immediate for $t \leq \frac{1}{2}$ and  follows from the
induction hypothesis for $t \geq \frac{1}{2}$.

By definition of strongly admissible simplices there exists a
vector of the canonical basis, say $e_q$, such that:
\begin{equation}\label{eq sim fortement adm preuve hom}(\sigma(x)-\sigma(x+\lambda e_q)) \in w,\end{equation} for
any $(x;\lambda)\in T_j ^q (k_v)$. This implies that
\begin{equation}\label{eq sigma w admi}(\sigma(0)-\sigma(e_q)) \in w.\end{equation}

 We distinguish two cases:

\medskip

\noindent \underline{\it First case:}
$s_\sigma(0)=s_\sigma(e_q)$. This implies 
 that for any $(x;\lambda) \in T_j ^q(k_v)$ we have
$$s_\sigma(x)=s_\sigma(x+\lambda e_q),$$
 and therefore \begin{equation}\label{eq
theta}|\theta(x)-\theta(x+\lambda e_q)|\leq |\nu_\sigma (x) -\nu
_\sigma (x +\lambda e_q)|.\end{equation} Note that if
$\xi'(\widetilde{\sigma}(x)) \in w$ then
$\xi'(\widetilde{\sigma}(x+\lambda e_q)) \in w$, which means that
in this case (\ref{eq preuve de (2)}) follows immediately from
(\ref{eq sim fortement adm preuve hom}) for $t \leq \frac{1}{2}$.
Otherwise $\xi'(\widetilde{\sigma}(x)) \notin w$ and then by
(\ref{eq sim fortement adm preuve hom}):
\begin{equation}\label{eq xi'}\frac{1}{2}\xi'(\widetilde{\sigma}(x))\leq \xi'(\widetilde{\sigma}(x+\lambda e_q)
)\leq 2\,\xi'(\widetilde{\sigma}(x)).\end{equation}

  Recall that the functions $\xi_i$ and $\xi_{i+1}$ are both  Lipschitz functions.
  Hence, if $\sigma$ is strongly admissible, for $t \leq \frac{1}{2}$ a straightforward computation shows that thanks to (\ref{eq theta}) and
  (\ref{eq xi'}) we have for any $(x;\lambda) \in T_j^q(k_v)$:
\begin{equation}\label{eq h fortment admissible}
(h_\sigma(x+\lambda e_q;t)-h_\sigma(x;t)) \in w.
\end{equation}

 For $t \geq \frac{1}{2}$, (\ref{eq preuve de (2)}) still  holds
thanks to the induction hypothesis and the Lipschitzness of
$\xi_i$ and $\xi_{i+1}$.

\medskip

\noindent \underline{\it Second case:} $s_\sigma(0)\neq
s_\sigma(e_q)$. In this case we observe that if $s_\sigma(0)$ is
$0$ then $$(\sigma_n(0)-\xi_i(\widetilde{\sigma}(0)) )\in w$$
which amounts to
$$d(\sigma(0); \Gamma_{\xi_i}) \in w,$$
(where $\Gamma _{\xi_i}$ denotes the graph of $\xi_i$). By
(\ref{eq sigma w admi}), this implies that  $d(\sigma(e_q);
\Gamma_{\xi_i})$ belongs to $ w$ and so
$$(\sigma_n(e_q)- \xi_i(\widetilde{\sigma}(e_q)) \in w.$$

As  $s_\sigma(e_q)$ is necessarily equal to $1$ we see that
$$\sigma_n(e_q)-\xi_{i+1}(\widetilde{\sigma}(e_q)) =0$$
so that $$\xi'(\widetilde{\sigma}(e_q)) \in w.$$ 
But, as the cell $E$ is $u$-admissible this implies that for {\it
any} $x \in E'$:
$$\xi'(x) \leq u.$$

This, together with the induction hypothesis, implies that
$h_\sigma$ satisfies (\ref{eq preuve de (2)}). This completes the
proof of $(2)$.
\medskip

\bigskip
It remains to prove $(3)$. First observe that all the $e_j$'s are
sent by $\varphi_\sigma$ onto vertices of $E$.  Note also that
$$\varphi_{\sigma}(x)=(\varphi_{\widetilde{\sigma}}(x);\xi_i(\varphi_{\widetilde{\sigma}}(x))+s_\sigma(x)\xi'(\varphi_{\widetilde{\sigma}}(x)))$$
and so, by the definition of the cells, the support of
$\varphi_\sigma$  lies in  cells of dimension at most $j$ of
$\mathcal{F}$. Moreover we just checked that (\ref{eq preuve de
(2)}) holds in any case. This implies that $\varphi_\sigma$ is
strongly admissible and therefore its support must lie in $Y_j$.
This completes the proof of the claim.
%
\end{proof}
\medskip
We are now able to  express the $v$-vanishing homology groups in
terms of the (usual) homology groups of some $v$-thin subsets
constituted by the $v$-thin cells of the barycentric subdivision
of some $L$-cell decompositions.

\begin{thm}\label{thm simplicial is singular}
For any $ X\subset R^n$  closed definable, there exist some
definable subsets of $X$: $$X_0 \subset \dots \subset X_{d+1}
=X_d$$ such that:
$$H_j ^v(X)\simeq Im (H_j(X_j)\to H_j(X_{j+1})) $$
(where the arrow is induced by inclusion and $Im$ stands for
image).
\end{thm}

\begin{proof}
 We start by defining inductively the subsets $ X_j $ 's. Set $X_0=\emptyset$ and assume that $X_0,\dots,X_{j-1}$  have already been
 defined.
  According to Proposition \ref{prop cell dec}, up to
a  bi-Lipschitz homeomorphism, we can assume that we have a
$u$-admissible $L$-cell decomposition compatible with $X_v $ and
$X_{j-1,v}$. Let $\mathcal{E}_j$ be the barycentric subdivision of
this cell decomposition and define $\Theta_j$ as  the union of all
the
   $(j;\N u)$-thin cells. There exists a $\la_R$-definable family $Y_{j}$ such that
$Y_{j,u}=\Theta_{j}$.  Now, thanks to Lemma \ref{lem N u egal N
t}, there exists $z$  in $v$, such that for any $t $ in $ v $
greater than $z$:
$$H_j^{  w}(Y_{j,t})\simeq H_j^{  u}(Y_{j,u}).$$  

Now  define $X_j$ as the subset of $R^n$ defined by a
$\la_R$-formula defining $Y_{j,t}$ for some $t \geq z$ in $v$.  If
$t$ is chosen large enough, $X_{j}$ is $ v$-thin. As bi-Lipschitz
homeomorphisms induce isomorphisms between the vanishing homology
groups, we identify subsets with their image so that, for
instance, we consider below the $X_{j,v}$'s and
$Y_{j,u}$ as  subsets of $X_v$. 
%
%
%

Consider the following diagram:
$$Im\{H_j(X_j) \to H_j(X_{j+1}) \} \overset{a}{\leftarrow} Im\{H_j^{  v}(X_{j}) \to H^{  v}_j(X_{j+1})\}\overset{b}{\rightarrow} H_j^v (X),$$
where again $a$ and $b$ are induced by the   inclusions of the
corresponding chain complexes. We shall show that $a$ and $b$ are
both isomorphisms.


\medskip

\noindent a is an isomorphism:    We have the following
commutative diagram:
\begin{center}
     \begin{picture}(-140,0)
      \put(-180,-3){$H_j ^v(X_{j})$}
        \put(-90,-3){$H_j^v (X_{j+1})$}
      \put(-180,-52){$H_j (X_{j})$}
          \put(-90,-52){$H_j  (X_{j+1})$}
      \put(-140,-50){\vector(3,0){40}}
      \put(-140,0){\vector(3,0){40}}
  \put(-165,-10){\vector(0,-1){30}}
      \put(-70,-10){\vector(0,-1){30}}
     \end{picture}
    \end{center}
\vskip 18mm where all the maps are induced by inclusion. By Remark
\ref{rmk x vthin vhom is hom}, the first  vertical arrow is an
isomorphism and the second
is one-to-one. This proves that $a$ is an isomorphism. 
\bigskip

\noindent $b$ is onto: Note that it is enough to prove that the
inclusion $X_j \to X$ induces an onto map between the
$v$-vanishing homology groups.

We have the following commutative diagram:

\begin{center}
     \begin{picture}(-140,0)

\put(103,-25){$\mbox{diag.} $\, 1.$$} \put(-130,-47){$ext$}
\put(-130,3){$ext$}
      \put(-180,-3){$H_j ^v(X_j)$}
      \put(-93,-3){$H_j (\Delta ^{ext}(X_{j,v}))$}
          \put(-93,-52){$H_{j+1} (\Delta ^{ext}(X_v))$}
          \put(-180,-52){$H_j ^v(X)$}
           \put(30,-3){$H_j ^w(X_{j,v})$}
          \put(30,-52){$H_j ^w(X_v)$}

      \put(-140,0){\vector(3,0){40}}
      \put(-140,-50){\vector(3,0){40}}
\put(-15,0){\vector(3,0){40}}
 \put(-15,-50){\vector(3,0){40}}

  \put(-165,-10){\vector(0,-1){30}}
       \put(40,-10){\vector(0,-1){30}}
     \end{picture}
    \end{center}
\vskip 20mm where the mapping $ext$, provided by extension of
chains, is an isomorphism (see section $3.1$).

By Lemma \ref{lem w, ext  et nu} the latter horizontal arrows are
isomorphisms as well. Therefore, it is enough to prove that the
map induced by inclusion $H_j ^w (X_{j,v}) \to H_j ^w (X_v)$ (the
last vertical arrow) is onto.

 For $t \geq z$ in $v$, let $\alpha$ and $\beta$ be the maps defined by
inclusion:
$$H_j ^w(Y_{j,t}) \overset{\alpha}{\to} H_j
(Y_{j,[z;u]})\overset{\beta}{\leftarrow} H_j ^u (Y_{j,u}).$$

 By Lemma \ref{lem N u egal N t}, $\alpha$ and $\beta$
are isomorphisms so that   $\gamma:= \beta^{-1} \alpha$ provides
the following commutative diagram:
\begin{center}
     \begin{picture}(-140,0)
\put(-173,-25){$\gamma$}
      \put(-190,-3){$H_j ^w(Y_{j,t})$}
        \put(-90,-3){$H_j ^w(X_v)$}
      \put(-190,-52){$H_j ^u(Y_{j,u})$}
          \put(-90,-52){$H_j ^u(X_v)$}
      \put(-140,-50){\vector(3,0){40}}
      \put(-140,0){\vector(3,0){40}}
  \put(-165,-10){\vector(0,-1){30}}
      \put(-75,-10){\vector(0,-1){30}}
     \end{picture}
    \end{center}
\vskip 18mm

By Lemma \ref{lem w, ext et nu} the second vertical arrow is onto.
Thus, it is enough to show that $H_j ^{ u} (Y_{j,u}) \to H_j ^{
u}(X_v)$ is onto. By construction,  $Y_{j,u}$ is the union of all
the $(j;\N u)$-thin closed cells of the barycentric subdivision of
$\mathcal{E}_j$. Note that it is enough to consider a chain
$\sigma \in\widehat{ Z}_j^w(X_v;\F)$ where $\F$ is the family
constituted by all the closure of the cells  of $\mathcal{E}_j$
(since the inclusion $\widehat{H}_j ^{ w} (X_{v};\F) \to H_j ^{
u}(X_v)$ is onto, thanks to Lemmas \ref{lem w, ext  et nu} and
\ref{lem strong adm is onto}).  By $(ii)$ of Proposition \ref{prop
homotopie},  there exists $\varphi_\sigma \in C_j^u(Y_{j,u})$ such
that $\sigma =\varphi_\sigma $ in $H_j^{ u}(X_v)$, as required.
\medskip

\noindent $b$ is one-to-one: Note that as diag. $1.$ holds for
$X_{j+1}$ as well (and the horizontal arrows are isomorphisms as
well), it is enough to show that the map induced by inclusion $$b'
:Im(H^w _j (X_{j,v}) \to H_j ^w (X_{j+1,v})) \to H_{j} ^w (X_v)$$
is one-to-one. Recall that by definition $X_{j+1,v}$ is
$Y_{j+1,t}$, for some $t$ and consider the following commutative
diagram:

%
%
%
\begin{center}
     \begin{picture}(-140,0)
\put(-3,-25){$\nu_t$}
 \put(-55,-47){$\nu_u$}
      \put(-125,-3){$H^w _j (X_{j,v})$}
          \put(-130,-50){$H_j ^u (Y_{j+1,u})$}
           \put(-20,-3){$  H_j ^w (Y_{j+1,t})$}
          \put(-20,-52){$H_j ^u(Y_{j+1,[z;u]})$}

\put(-75,0){\vector(3,0){50}}
 \put(-75,-50){\vector(3,0){50}}

      \put(-100,-10){\vector(0,-1){30}}
       \put(-5,-10){\vector(0,-1){30}}
     \end{picture}
    \end{center}
    \vskip2cm
    where again $\nu_u$ and $\nu_t$ are induced by the respective
    inclusions. By Lemma \ref{lem N u egal N t} these maps are
    one-to-one.

 This implies that  we have the following commutative  diagram:
\begin{center}
     \begin{picture}(-140,0)
 \put(-15,3){$b'$}
 \put(-15,-47){$b''$}
\put(-111,-25){$\mu$}
      \put(-170,-3){$Im(H^w _j (X_{j,v}) \to H_j ^w (Y_{j+1,t}))$}
          \put(-170,-50){$Im(H^w _j (X_{j,v}) \to H_j ^u (Y_{j+1,u}))$}
           \put(30,-3){$H_j ^w(X_{v})$}
          \put(30,-52){$H_j ^u(X_v)$}

\put(-25,0){\vector(3,0){50}}
 \put(-25,-50){\vector(3,0){50}}

      \put(-100,-10){\vector(0,-1){30}}
       \put(40,-10){\vector(0,-1){30}}
     \end{picture}
    \end{center}
\vskip 20mm where all the horizontal arrows  are induced by the
corresponding inclusions and $\mu$ is induced by  the restriction
of $\nu_u ^{-1}\nu_t $. Since $\mu$ is one-to-one, it is enough to
show that $b''$ is one-to-one.

To check that $b''$ is one-to-one, take $\sigma$ in $Z_j^{
w}(X_{j,v})$ which bounds a chain of $C_{j+1} ^u (X_v)$. As the
inclusion $H_j ^w (X_v) \to H_j ^u (X_v)$ is an isomorphism, there
exists $\tau$ in $C_{j+1} ^{w}(X_v)$ such that $\sigma =\partial
\tau$.  Consider a $w$-admissible rectilinearization of $|\tau|$,
$|\sigma|$ and $\F$ where $\F$ is the family constituted by  the
closure the cells of the barycentric subdivision of
$\mathcal{E}_{j+1}$. The chain $\sigma$ is equal in $H_j
(X_{j,v})\simeq H_j^w(X_{j,v})$ (see Remark \ref{rk bar subd
rect}) to a simplicial chain $\sigma'$ which is strongly
$w$-admissible and compatible with $\F$. The class of the chain
$\sigma$ is zero in $H_j(|\tau|)$ and therefore $\sigma'$ bounds a
simplicial chain $\tau'$ which is also strongly $w$-admissible
(again by Remark \ref{rk bar subd rect}). 

 By construction,    $X_{j,v}$ is a union
of cells of $\mathcal{E}_{j+1}$ and  the union of all the closure
of  the cells of  dimension $(j+1)$ of the barycentric subdivision
of $\mathcal{E}_{j+1}$ which are $(j+1;\N u)$-thin is precisely
$Y_{j+1,u}$. Therefore we may apply Proposition \ref{prop
homotopie} to $X_v$.
 This provides a map $\varphi:
 \widehat{C} _j ^w (X_v;\F) \to\widehat{C}_j ^u (Y_{j+1,u};\F)$
 such that
   $$\partial \varphi _{\tau'}=\varphi_{\partial \tau'}=\varphi_{\sigma'}.$$

   As by $(iii)$ of this proposition
  $\sigma'=\varphi_{\sigma'}$
   in $H_j ^u (X_{j,v})$, this implies that the class of $\sigma$ is zero in
 $H_j^u
 (Y_{j+1,u})$ and yields that $b''$ is one-to-one.
\end{proof}
\medskip

\begin{cor}\label{cor groups finitely gen}
For any closed definable subset $X$, the vanishing homology groups
$H_j ^v(X)$ are finitely generated.
\end{cor}

 Note that the above corollary
 enables us to define an Euler characteristic which is a definable metric invariant by setting:
 $$\chi_v (X) :=\sum_{i=1} ^\infty (-1)^i \dim H_i
^v (X).$$

 This invariant for definable subsets of $R^n$
gives rise to a metric  for definable families or for germs of
definable sets (see example \ref{ex fam} and section $4$ below).

\begin{rk}The hypothesis closed is assumed for convenience. We could shrink an open tubular neighborhood of radius $z \in v$
 of the points lying in the closure
but not in $X$
    so
that we would have a deformation retract of our set onto the
complement of this neighborhood which is very close to the
identity, and hence which preserves thin subsets, identifying the
vanishing homology groups of our given set with those of a closed
subset.
 \end{rk}

\section{Local invariants for  singularities.} In \cite{v2}, we
introduced the link for a semi-algebraic metric space. Let us
recall its definition. We recall that we denote by  $k(0_+)$ the
field of algebraic Puiseux series endowed with the order that
makes the indeterminate $T$ positive and  smaller than any real
number. We denote by  $d$  the Euclidian distance. Given the germ at $0$ of a semi-algebraic set $X$   let: 
$$L_X:=\{x \in X_{k(0_+)}: d(x;0)=T\} $$
where $T \in k(0_+)$ is the indeterminate and $X_{k(0_+)}$ the
extension of $X$ to $k(0_+)$.

In \cite{v2} we proved that the  set $L_X$ is a metric invariant
which characterizes the metric type of the singularity in the
sense that: \begin{thm}\label{thm link}\cite{v2} $L_X$ is
semi-algebraically bi-Lipschitz homeomorphic to $L_Y$ iff $Y$ is
semi-algebraically bi-Lipschitz homeomorphic to $X$.\end{thm}

This theorem admits the following immediate corollary.

\begin{cor}
For any convex subgroup  $v \subset k(0_+)$, the groups $H_j ^{v}
(L_X)$ are semi-algebraic bi-Lipschitz invariants of $X$.
\end{cor}

Note that by Corollary \ref{cor groups finitely gen} these groups
are finitely generated and that $\chi_v(L_X)$ is a semi-algebraic
bi-Lipschitz invariant of the germ $X$.

 \medskip

\begin{rk}
We assumed in this section that $X$ is a semi-algebraic set
because this was the setting of \cite{v2}. Nevertheless, the main
ingredient of  the proof of Theorem \ref{thm link} is Theorem
$5.1.3$ of \cite{v1}. As this theorem holds over  any polynomially
bounded o-minimal structure, the above corollary is still true in
this setting as well. The metric type of the link $L_X$ may fail
to be a metric invariant of the singularity when the set is
definable in a non-polynomially bounded o-minimal structure as it
is shown by the following example.
 \end{rk}
\begin{ex}
Let $X:=\{(x;y) \in \R^2: |y| = e^{\frac{-1}{x^2}}\}$ and
$Y=\{(x;y) \in \R^2: |y| = e^{\frac{-2}{x^2}}\}$. Note that $X$
and $Y$ are both definable in the $ln-exp$ structure (see
\cite{vds}, \cite{rl}, \cite{w}). Furthermore $X$ and $Y$ are
definably bi-Lipschitz homeomorphic. However the links of $X $ and
$Y$ are constituted by two points of $k_{0_+}^2$ (where $k_{0_+}$
is the corresponding residue field) whose respective distances are
clearly not equivalent.

Note that a revolution of these subsets about the $x$-axis
provides two subsets whose links  have different vanishing
homology groups (for a suitable velocity).
\end{ex}

\section{Some examples.}We give two examples of computations of the homology groups.
 It is convenient to develop ad hoc techniques to compute the
 homology groups such as the  excision
 property.

\subsection{ The excision property.}
 It follows from the  definition that we may  have
 $c+c'$ in $C_j ^v$ although  neither $c$ nor $c'$ belong to this
 set. This is embarrassing since it makes it impossible the
 splitting of a chain of $X$ into a chain of $X\setminus A$ plus a
 chain of $A$, which is crucial for the  excision property. To overcome this difficulty we are going to
 consider
 more chains. This will {\it not} affect the
 resulting
 homology groups.


 We defined the vanishing homology groups by
 requiring for a chain $\sigma$ that $|\sigma|$ and $|\partial
 \sigma|$ to be  both $(j;v)$-thin.  We may work with another chain complex.

 Let $A$ and $X$ be closed definable subsets of $R^n$ with $A \subset X$  and denote by $\F$ the pair $\{X
\setminus Int(A);A\}$ where $Int (A)$ is the interior of $A$. Let
$\Delta^v _j(X)$ the subset of $C ^v _j(X;\F)$ constituted by the
  $j$-chains having a $(j;v)$-thin support. Of course, such a family of modules  is not preserved by the
 boundary operator but, if we want to have a chain complex, we may add the boundaries by setting:
 $$\Delta'^v _j(X):=\Delta_j ^v(X) +\partial \Delta_{j+1}^v(X).$$
 This provides a chain complex
with  obviously $H_j(\Delta'^v (X))=H_j^v(X)$.

The inconvenient is that we are going to work with non admissible
chains but the advantage is that we have now more freedom to work
since we have more chains. For instance if $(c_1+c_2) \in
\Delta'^v _j (X)$ then $c_1$ and $c_2$ both belong to $\Delta'^v
_j (X)$.

\medskip
To state the excision property we need to introduce the homology
groups of a pair. For this purpose, we first set:$$ \Delta_j ^v(X
/A):=\{c \in \Delta_{j} ^v(X  ):  (\partial c-\partial_A c) \in
\Delta_{j-1} ^v(X ) \},$$ where $\partial_{ A}$ takes the boundary
and projects it onto $C_j ( A)$.

 Define also
$$\Delta_{j}^{v,X}(A):= \Delta_j ^v( A) + \partial_{ A} \Delta_{j+1} ^v(X/ A ).  $$

 First observe that by definition
if $c \in \Delta_{j+1}^v(X/A )$ then $$\partial_A c \in \Delta_{j}
^v(X ) +
\partial \Delta_{j+1} ^v(X ).$$

Therefore, by definition of $\Delta^{v,X}_j$ we get
$$\Delta_j ^{v,X}(A ) \subset \Delta_j ^v(X )+ \partial \Delta_{j+1} ^v(X ) =\Delta'^v_j(X ).
$$

Thus, we may set $$\Delta^v _j(X;A ):=\frac{\Delta'^v_j (X
)}{\Delta^{v,X}_j(A )},$$

and $$H_j ^v(X;A):=H _j (\Delta^v(X;A )).$$
\begin{rk}\label{rk hom pair X thin}
If $X$ is a $v$-thin set {\it of dimension $j$} then $H_j ^v(X;A)=
H_j (X;A)$ (see Remark \ref{rmk x vthin vhom is hom}).
\end{rk}

Let $i:\Delta _j ^{v,X}(A ) \to \Delta'^v _j(X )$ be the
inclusion. Clearly, we
 have the following  exact sequences:
\begin{equation}\label{lem exact sequence pair}
  0 \to
 \Delta_j ^{v,X}(A  ) \overset{i}{\to} \Delta_j '^v (X  )
 \overset{q}{\to} \Delta'^v_j (X;A  )\to 0,
\end{equation}
(where  $q$ is the quotient map)
\bigskip
and therefore we get the following long exact sequence:
$$...\to H_j ^v(\delta_X A) \to H_j ^v(X) \to H_j ^v (X;A) \to H ^v _{j-1} (\delta_X A) \to  ...$$

\medskip

\begin{rk}
We could have defined the homology groups of a pair by $H_j
^v(X;A):= H_j^v(C^v(X;A))$ where $C^v(X;A):=\frac{C_j
^v(X)}{C_j^v(A)}$, and of course the latter exact sequence would
hold for $H_j ^v(A)$ (instead of $H_j ^v (\delta_X A)$). However
the excision property would not hold.
\end{rk}

As we said, if $(c + c')$ belongs to $\Delta_j ^v(X)$ then $c$ and
$c'$ both belong to $C_j ^v (X)$. Therefore, the excision property
holds for $H_j ^v (X;A)$. Let $(X;A)$ and $W$ be definable such
that $W$ lies in the interior of $A$. Then for any $j$:
\begin{equation}\label{eq excision property}
H_j ^v(X;A) = H_j ^v (X \setminus W;A \setminus W).
\end{equation}

\medskip

\subsection{Two examples.} We are going to compute the vanishing
homology groups on two examples which are semi-algebraic families.
Let us  take  $\Q$ as our coefficient group.

\begin{ex} We first compute the homology groups on the example sketched on fig 1. We consider two spheres from which we shrink a little disk which
collapses into a point and which intersects along the boundaries
of these disks.

Let $$X(\varepsilon):=\{(x;y;z) \in k(0_+) ^3:
(x-\varepsilon(1-T^4))^2+y^2+z^2=1,
 \varepsilon x \geq 0\}$$
for $\varepsilon =\pm 1$. Then let $X:=X(1) \cup X(-1)$ and
$A=X(1) \cap X(-1)$.

 Let us simply consider the velocity $T^2$. The
computation could actually be carried out for any velocity. Since
the set $A $ is $\N T^2$-thin we have: $$H_1 ^{T^2}(\delta_X A
)=H_1 ^{T^2}(A )=H_1(A)=\Q,$$ and  $H_0^{T^2}(\delta_X A )=0$.

Note that, thanks to the excision property, we have:
$$H_1 ^v (X;A) \simeq H_1 ^v(X(1);A) \oplus  H_1 ^v (X(-1);A). $$
If we add the disk  $$D=\{(x;y;z) \in k(0_+) ^3:
(x-\varepsilon(1-T^4))^2+y^2+z^2=1,
 \varepsilon x \leq 0\}$$ to $X(1)$,
we get the sphere $S^2$. Thus, by the excision property, $$H_1
^{T^2} (X(1);A) \simeq H_1 ^{T^2} (S^2;D)=0,$$ and so $H_1 ^{T^2}
(X;A)=0$.
 Examining the
exact sequence of the pair $(X;A)$ 
 we see that: $$H_1 ^{T^2}(X)\simeq H_1 ^{T^2}(\delta_X A)\simeq \Q.$$
 Observe also that we  have: $H_2
^{T^2} (X)\simeq 0$ and $H_0 ^{T^2}(X)\simeq 0$.
 \end{ex}

\medskip

We end by computing the vanishing homology groups of
Birbrair-Goldshtein examples (compare with \cite{bb1} section 7.)
.

\begin{ex}\label{ex tore}
Let $X$ be the set defined by (\ref{eq birgol})   assume that
$p<q$. Let us compute for instance the vanishing homology groups
for the velocity $T^q$. We could use here the excision property
and follow the classical methods for computing the homology groups
of the torus but it is actually simpler to derive it from the
classical homology groups of $X$ since it is $\N T^q$-thin. This
implies that the inclusion $H^{T^q} _2 (X) \to H_2 (X)$ is an
isomorphism and that the inclusion $H_1 ^{T^q} (X) \to H_1  (X)$
is one-to-one. Therefore $$H^{T^q} _2 (X) \simeq \Q$$ and $\dim
H^{T^q} _1 (X) \leq 2$. Actually,  one generator of $ H_1 (X)$
 has  a representant with  $T^q$-thin support and every $1$-chain representing a different class has a support whose length is clearly
 bigger than $T^p$. This  proves that  $\dim H^{T^q} _1
(X) = 1$.

%
%


\end{ex}

\noindent{\bf Acknowledgements.} The author is very grateful to
Pierre Milman for valuable discussions on related questions and
would like to thank the Department of Mathematics of the
University of Toronto where this work has been carried out.

\end{document}